\documentclass[a4paper,draft]{amsproc}
\usepackage{amssymb}
\usepackage[hyphens]{url} 

\theoremstyle{plain}
 \newtheorem{thm}{Theorem}[section]

 \newtheorem{cor}{Corollary}[section]
\theoremstyle{definition}
 \newtheorem{exm}{Example}[section]
 \newtheorem{dfn}{Definition}[section]
\theoremstyle{remark}
 
 \numberwithin{equation}{section}

\renewcommand{\leq}{\leqslant}
\renewcommand{\geq}{\geqslant}

\title[Statistical convergence in $g$-metric spaces]{Statistical convergence in $g$-metric spaces}



\author{Rasoul Abazari} 

\address{ 
Department of Mathematics  \\ 
, Ardabil branch, Islamic Azad University  \\ 
Ardabil\\
Iran}
\email{r.abazari@iauardabil.ac.ir,  rasoolabazari@gmail.com}

\begin{document}

{\begin{flushleft}\baselineskip9pt\scriptsize

\end{flushleft}}
\vspace{18mm} \setcounter{page}{1} \thispagestyle{empty}

\begin{abstract}
The purpose of this paper is to define statistically convergent sequences with respect to the metrics on generalized metric spaces ($g$-metric spaces) and investigate basic properties of this statistical form of convergence.

\end{abstract}

\maketitle

\section{Introduction and Preliminaries}
There are numerous ways to generalize the notion of distance function \cite{Agarwal,M. A. Khamsi}. Among them, the concept of G-metric space that has been studied by Mustafa \cite{Z Mustafa B Sims} is a different generalization of the ordinary metric. Metrics in this space are distance between three points. For more generalization, Choi et al. \cite{H.Choi} introduced $g$-metric with degree $n$, that is a distance between $n+1$ points. In this paper, convergence of sequences from the topological properties of $g$-metric space will be discussed in terms of extending it to statistical forms.

The idea of statistical convergence was first expressed in 1935 by Zygmund \cite{A. Zygmund}. The formal concept of statistical convergence was introduced by Steinhaus \cite{H. Steinhaus} and Fast \cite{H. Fast} in 1951. Afterwards, Shoenberg reintroduced it in 1959 \cite{I.J. Schoenberg}.\\ Since then, the properties of statistical convergence have been studied by different mathematicians and applied in several area such as approximation theory \cite{O. Duman M K. Khan and C. Orhan}, trigonometric series \cite{A. Zygmund}, finitely additive set functions \cite{J. Connor and J. Kline}, Stone-Chech compactification \cite{J. Connor and M.A. Swardson}, Banach spaces \cite{J. Connor Kadets}, probability theory \cite{J.A. Fridy and M.K. Khan} and summability theory \cite{J. Connor,J.A. Fridy,J.A. Fridy and C. Orhan,E.Savas}.

The main goal of this paper is introducing the definition of statistically convergent sequence and studying its properties in $g$-metric spaces. 
First, lets bring some basic concepts that are needed in the rest of this paper.
\begin{dfn}
	\cite{Z Mustafa B Sims} Let $X$ be a nonempty set and Let $G:X\times X \times X \to \mathbb{R}^+$, be a function satisfying:
	\item[1)]
	$G(x,y,z)=0$ if $x=y=z$,
	\item[2)]
	$0<G(x,x,y);$ for all $x,y \in X$ , with $x\neq y$,
	\item[3)]
	$G(x,x,y)\leq G(x,y,z)$ , for all $x,y,z\in X$ with $z\neq y$,
	\item[4)]
	$G(x,y,z)=G(x,z,y)=G(y,z,x)=...$ , (symmetry in all three variables),
	\item[5)]
	$G(x,y,z)\leq G(x,a,a)+G(a,y,z)$, for all $x,y,z,a \in X$ , (rectangle inequality)
\end{dfn}
The function $G$ is called a generalized metric or $G$-metric on $X$, and the pair $(X,G)$ is a $G$-metric space.\\
The following definition is an extension of above space with degree $l$.

\begin{dfn}
\cite{H.Choi} Let $X$ be a nonempty set. A function $g:X^{l+1}\longrightarrow \mathbb{R}_+$ is called a $g$-metric with order $l$ on $X$ if it satisfies the following conditions:
\item[g1)]
$g(x_0, x_1, ..., x_l)=0 \quad \text{if and only if} x_0=x_1=...=x_l$,
\item[g2)]
$g(x_0, x_1, ..., x_l)=g(x_{\sigma(0)}, x_{\sigma(1)},...,  ...,x_{\sigma(l)})$ for permutation  $\sigma$ on $\{0, 1, ..., l\}$,
\item[g3)]
$g(x_0, x_1, ..., x_l)\leq g(y_0, y_1, ..., y_l)$ for all $(x_0, x_1, ..., x_l), (y_0, y_1, ..., y_l)\in X^{l+1}$ with $\{x_i: i=0, 1, ..., l\}\subseteq \{y_i: i=0, 1, ..., l\},$
\item[g4)]
For all $x_0, x_1, ..., x_s, y_0, y_1, ..., y_t, w\in X$ with $s+t+1=l$
$$g(x_0, x_1, ..., x_s, y_0, y_1, ..., y_t)\leq g(x_0, x_1, ..., x_s, w, w, ..., w)+g(y_0, y_1, ..., y_t, w, w, ..., w)$$
The pair $(X,g)$ is called a $g$-metric space. It is noteworthy that, if $l=1$ (resp. $l=2$), then it is equivalent to ordinary metric space (resp. $G$-metric space).
\end{dfn}
The following theorem will be needed in the main results.
\begin{thm}
\cite{H.Choi} Let $g$ be a $g$-metric with order $l$ on a nonempty set $X$. The followings are true:
\item[1)]
$g(\underbrace{x, ...,x}_\text{s times}, y, ..., y)\leq g(\underbrace{x, ...,x}_\text{s times}, \omega, ..., \omega)+ g(\underbrace{\omega, ...,\omega}_\text{s times}, y, ..., y)$,
\item[2)]
$g(x,y, ..., y)\leq g(x, \omega, ..., \omega)+g(\omega, y, ..., y)$
\item[3)]
$g(\underbrace{x, ...,x}_\text{s times}, \omega, ..., \omega)\leq sg(x, \omega, ..., \omega)$
and\\
\hspace{5 mm} $g(\underbrace{x, ...,x}_\text{s times}, \omega, ..., \omega)\leq (l+1-s)g(\omega, x, ..., x)$,
\item[4)]
$g(x_0, x_1, ..., x_l)\leq \sum_{i=0}^{n}g(x_i, \omega, ..., \omega),$
\item[5)]
$|g(y, x_1, ..., x_l)-g(\omega, x_1, ..., x_l)|\leq \max \{g(y, \omega, ..., \omega), g(\omega, y, ..., y)\}$,
\item[6)] 
$|g(\underbrace{x, ...,x}_\text{s times}, \omega, ..., \omega)-g(\underbrace{x, ...,x}_\text{s' times}, \omega, ..., \omega)|\leq |s-s'|g(x, \omega, ..., \omega),$

\item[7)]
$g(x, \omega, ..., \omega)\leq(1+(s-1)(l+1-s)g(\underbrace{x, ...,x}_\text{s times}, \omega, ..., \omega).$
\end{thm}

\begin{dfn}
Let $(X,g)$ be a $g$-metric space. Let $x\in X$ be a point and $\{x_k\}\subseteq X$ be a sequence.
\item[1)]
$\{x_k\}$ $g$-converges to $x$, if for all $\epsilon>0$ there exists $N\in \mathbb{N}$ such that 
$$i_1, ..., i_l\geq N\Longrightarrow g(x, x_1, ..., x_l)<\epsilon.$$

\item[2)]
$\{x_k\}$ is said to be $g$-Cauchy if for all $\epsilon>0$ there exists $N\in\mathbb{N}$ such that 
$$i_0, i_1, ..., i_l\geq N\Longrightarrow g(x_{i_0}, x_{i_1}, ..., x_{i_l})<\epsilon.$$

\item[3)]
$(X, g)$ is complete if every $g$-Cauchy sequence in $(X, g)$ is $g$-convergent in $(X, g)$.
\end{dfn}

\section{Main Results}
In this section, the definition of statistical convergence of sequences in $g$-metric spaces is introduced and some basic properties are studied. \\
For a set $K$ of positive integers, the asymptotic or natural density is defined as follows,
\begin{equation*}
	\delta(K)=\lim_n\frac{1}{n}|\{k\leq n : k\in K\}|,
\end{equation*}
where $|\{k\leq n : k\in K\}|$ denotes the number of elements of $K$ not exceeding $n$.

\begin{dfn}
		\cite{H. Fast} The sequence $\{x_k\}$ is said to be statistically convergent to $x$, if for every $\epsilon>0$,
	\begin{equation*}
		\lim_n\frac{1}{n}|\{k\leq n : |x_k-x|< \epsilon\}|=1.
	\end{equation*}
\end{dfn}

\begin{dfn}
The sequence $\{x_k\}$ is said to be statistical Cauchy sequence if for every $\epsilon>0$, there exists a number $N$ depending on $\epsilon$ such that,
	\begin{equation*}
	\lim_n\frac{1}{n}|\{k\leq n : |x_k-x_N|< \epsilon\}|=1.
\end{equation*}
\end{dfn}
For more information about properties of statistical convergence,  \cite{H. Fast,J.A. Fridy and C. Orhan,J.A. Fridy,O. Duman M K. Khan and C. Orhan} can be addressed. \\

Now, the main definitions of this paper are ready to be given.
\begin{dfn}
Let $l\in \mathbb{N}$, $A\in \mathbb{N}^l$ and 
$$A(n)=\{i_1, i_2, ..., i_l\leq n: (i_1, i_2, ..., i_l)\in \ A  \},$$ 
then 
\begin{equation*}
\delta_l(A):=\lim_{n\to \infty}\frac{l!}{n^l}|A(n)|,
\end{equation*}
is called {\em $l$-dimensional asymptotic (or natural) density} of the set $A$.
\end{dfn}

\begin{dfn}
Let $\{x_n\}$ be a sequence in a $g$-metric space $(X,g)$ such that 
$$g: X^{l+1}\longrightarrow \mathbb{R}^+$$ 
\item[i)]
$\{x_n\}$ statistically converges to $x$, if for all $\epsilon>0$,
\begin{equation*}
\lim_{n\to \infty}\frac{l!}{n^l}|\{(i_1, i_2, ..., i_l)\in \mathbb{N}^l:\quad i_1, i_2, ..., i_l\leq n,\quad g(x, x_{i_1}, x_{i_2},..., x_{i_l})<\epsilon\}|=1.
\end{equation*}
And denoted by; $gs-\lim_{n\to \infty} x_n=x$ or $x_n\xrightarrow{gs}x$\\

\item[ii)] 
$\{x_n\}$ is said to be statistical $g$-Cauchy, if for $\epsilon>0$, there exists   
$i_{\epsilon}\in \mathbb{N}$ such that 
\begin{equation*}
\lim_{n\to \infty}\frac{l!}{n^l}|\{(i_1, i_2, ..., i_l)\in \mathbb{N}^l:\quad i_1, i_2, ..., i_l\leq n,\quad g(x_{i_{\epsilon}}, x_{i_1}, x_{i_2},..., x_{i_l})<\epsilon\}|=1.
\end{equation*}
\end{dfn}

\begin{thm}
In g-metric spaces, every convergent sequence is statistically convergent.
\end{thm}
\begin{proof}
Let $\{x_n\}$ be a sequence in g-metric space $(X,g)$ such that converges to $x$.  
For $\epsilon>0$ there exist $n_0\in\mathbb{N}$ such that for all $i_1, i_2, ..., i_l\geq n_0$, $$g(x, x_{i_1}, x_{i_2},..., x_{i_l})<\epsilon.$$
 Set 
$$A(n):=\{(i_1, i_2, ..., i_l)\in \mathbb{N}^l:  i_1, i_2, ..., i_l\leq n,\ g(x, x_{i_1}, x_{i_2},..., x_{i_l})<\epsilon\},$$
then
$$|A(n)|\geq \binom{n-n_0}{l},$$
and 
\begin{equation*}
\lim_{n\to\infty}\frac{l!|A(n)|}{n^l}\geq\lim_{n\to\infty}\frac{l!}{n^l} \binom{n-n_0}{l}=1,
\end{equation*}
so 
\begin{equation*}
gs-\lim_{n\to \infty} x_n=x.
\end{equation*}
\end{proof}
The following example shows that the converse of above theorem is not valid.
\begin{exm}
Let $X=\mathbb{R}$ and $g$ be the metric as follows;

$$g:\mathbb{R}^3\longrightarrow\mathbb{R}^+,$$
$$g(x,y,z)= max\{|x-y|, |x-z|, |y-z|\}.$$
Consider the following sequence,

\begin{equation*}
	x_k=\left\{
	\begin{array}{rl}
		k & \text{if $k$ is square} \quad \text {}    .\\
		0 & \quad\quad\text{o.w}
		
	\end{array} \right.
\end{equation*}
 It is clear that $\{x_k\}$ is statistically convergent while it is not convergent normally. 
\end{exm}
Following theorem shows that statistical limit in $g$-metric space is unique.
\begin{thm}
Let $\{x_n\}$ be a sequence in g-metric space $(X,g)$ such that\\ $x_n\xrightarrow{gs}x$ and $x_n\xrightarrow{gs}y$, then $x=y$.
\end{thm}
\begin{proof}
 For arbitrary $\epsilon>0$, set
$$A(\epsilon):=\{(i_1, i_2, ..., i_l)\in \mathbb{N}^l:\quad g(x, x_{i_1}, x_{i_2},..., x_{i_l})\geq\frac{\epsilon}{2l}\},$$
$$B(\epsilon):=\{(i_1, i_2, ..., i_l)\in \mathbb{N}^l:\quad g(y, x_{i_1}, x_{i_2},..., x_{i_l})\geq\frac{\epsilon}{2l}\},$$
Since $x_n\xrightarrow{gs}x$ and $x_n\xrightarrow{gs}y$,  therefore $\delta_l(A(\epsilon))=0$ and $\delta_l(B(\epsilon))=0$.\\
Let $C(\epsilon):=A(\epsilon)\cup B(\epsilon)$, then $\delta_l(C(\epsilon))=0$, hence $\delta_l(C^c(\epsilon))=1.$\\
Suppose $(i_1, i_2, ..., i_l)\in C^c(\epsilon)$, then by Theorem 1.1 we have; 
\begin{align*}
g(x, y, y,..., y)&\leq g(x, x_m, x_m, ..., x_m)+g(x_m, y,y,..., y)\\
 &\leq  g(x, x_m, x_m, ..., x_m)+l(g(y, x_m, x_m, ..., x_m))\\ 
 &\leq g(x, x_{i_1}, x_{i_2},..., x_{i_l})+lg(y, x_{i_1}, x_{i_2},..., x_{i_l})\\
 &\leq 
l(g(x, x_{i_1}, x_{i_2},..., x_{i_l})+g(y, x_{i_1}, x_{i_2},..., x_{i_l}))\\
&\leq l(\frac{\epsilon}{2l}+\frac{\epsilon}{2l})=\epsilon.
\end{align*}

Since $\epsilon>0$ is arbitrary, we get
$$g(x, y, y,..., y)=0,$$
therefore $x=y$.
\end{proof}
\begin{dfn}
A set $A=\{n_k: k\in \mathbb{N}\}$ is said to be statistically dense in $\mathbb{N}$, if the set 
$$A(n)=\{(i_1, i_2, ..., i_l)\in \mathbb{N}^l: i_j\in A\quad ,\quad i_1, i_2, ..., i_l\leq n \},$$ 
has asymptotic density 1. i.e.,
\begin{equation*}
\delta_l (A)= \lim _{n \to \infty}\frac{l!|A(n)|}{n^l}=1.
\end{equation*}

\end{dfn}
\begin{dfn}
A subsequence $\{x_{n_k}\}$ of a sequence $\{x_n\}$ in a $g$-metric space $(X, g)$ is statistically dense, if the index set $\{n_k: k\in \mathbb{N}\}$ is statistically dense subset of $\mathbb{N}$. i.e., \\
$$\delta_l (\{n_k; k\in \mathbb{N}\})=1.$$
\end{dfn}

Equivalency in the following theorem has been studied for cone metric space in \cite{K.Li}, now we prove it on $g$-metric space.
 
\begin{thm}
Let $\{x_n\}$ be a sequence in a $g$-metric space $(X, g)$. Then the followings are equivalent.
\item[1)]
$\{x_n\}$ is statistically convergent in $(X,g)$.

\item[2)]
There is a convergent sequence $\{y_n\}$ in $X$ such that $x_n=y_n$ for almost all $n\in \mathbb{N}$.

\item[3)]
There is a statistically dense subsequence $\{x_{n_k}\}$ of $\{x_n\}$ such that $\{x_{n_k}\}$ is convergent.

\item[4)]
There is a statistically dense subsequence $\{x_{n_k}\}$ of $\{x_n\}$ such that $\{x_{n_k}\}$ is statistically convergent.
\end{thm}
\begin{proof}
$(1\Longrightarrow 2)$\\
Let $\epsilon>0$ and $\{x_n\}$ be a sequence that statistically converges to $x\in X$. i.e., 
\begin{equation*}
\lim_{n\to \infty}\frac{l!}{n^l}|\{(i_1, i_2, ..., i_l)\in \mathbb{N}^l: i_1, i_2, ..., i_l\leq n,\ g(x, x_{i_1}, x_{i_2},..., x_{i_l})<\epsilon\}|=1.
\end{equation*}
For every $k\in \mathbb{N}$, there exist $n_k\in \mathbb{N}$, such that for every $n>n_k$, 
\begin{equation*}
\frac{l!}{n^l}|\{(i_1, i_2, ..., i_l)\in \mathbb{N}^l: i_1, i_2, ..., i_l\leq n,\ g(x, x_{i_1}, x_{i_2},..., x_{i_l})<\frac{1}{2^k}\}|>1-\frac{1}{2^k}.
\end{equation*}
We can choose $\{n_k\}$ as an increasing sequence in $\mathbb{N}$.
Define $\{y_m\}$ as follows
\begin{equation*}
y_m=
\begin{cases}
x_m, \quad 1\leq m\leq n_1\\
x_m,  \quad  n_k< m\leq n_{k+1},\  i_1, i_2, ..., i_{l-1}\leq n_{k+1},\quad g(x, x_{i_1}, x_{i_2},..., x_{i_l})<\frac{1}{2^k}\}|>1-\frac{1}{2^k}\\
x,  \qquad \text{otherwise}
\end{cases}
\end{equation*}
Choose $k\in \mathbb{N}$ such that $\frac{1}{2^k}<\epsilon$. It is clear that $\{y_m\}$ converges to $x$.
Fix $n\in \mathbb{N}$, for $n_k<n\leq n_{k+1}$, we have,
\begin{align*}
\{(i_1, i_2, ..., i_l)&\in \mathbb{N}^l: i_1, i_2, ..., i_l\leq n ; x_{i_j}\neq y_{i_j}\}\\
&\subseteq \{(i_1, i_2, ..., i_l)\in \mathbb{N}^l: i_1, i_2, ..., i_l\leq n\}\\
&-\{(i_1, i_2, ..., i_l)\in \mathbb{N}^l: i_1, i_2, ..., i_l\leq n_k,\  g(x, x_{i_1}, x_{i_2},..., x_{i_l})<\frac{1}{2^k}\}.
\end{align*}
So
\begin{align*}
&\lim_{n \to \infty}\frac{l!}{n^l}|\{(i_1, i_2, ..., i_l)\in \mathbb{N}^l: i_1, i_2, ..., i_l\leq n ; x_{i_j}\neq y_{i_j}\}|\leq \lim_{n\to\infty}\frac{l!}{n^l}\binom{n}{l}\\
&-\lim_{n \to \infty}\frac{l!}{n^l}|\{(i_1, i_2, ..., i_l)\in\mathbb{N}^l: i_1, i_2, ..., i_l\leq n ;g(x, x_{i_1}, x_{i_2},..., x_{i_l})<\frac{1}{2^k}\}|\\
&\leq 1-(1-\frac{1}{2^k})=\frac{1}{2^k}<\epsilon.
\end{align*}
Hence 
$$\delta_l (\{(i_1, i_2, ..., i_l)\in \mathbb{N}^l: x_{i_j}\neq y_{i_j}\})=0 \qquad \text{\em (almost all)}$$

$(2\Longrightarrow3)$\\
Suppose that $\{y_n\}$ be a convergent sequence in $X$ such that $x_n=y_n$ for almost all $n\in \mathbb{N}$. Set $A=\{n\in \mathbb{N}: x_n=y_n\}$. Since $x_n=y_n$ for almost all $n$ , hence $\delta_l(A)=1$ and therefore  $\{y_n; n\in A \}$ is convergent and statistically dense subsequence of $\{x_n\}$.

$(3\Longrightarrow4)$\\
It is direct consequence of Theorem 2.1.

$(4\Longrightarrow1)$\\
Suppose $\{x_{n_k}\}$ be a statistically dense subsequence  of the sequence $\{x_n\}$ such that statistically converges to $x\in X$, i.e.,  $gs-\lim_{k\to\infty}x_{n_k}=x\in X$ and set $A=\{n_k; k\in\mathbb{N}\}$, then $\delta_l(A)=1$. For $\epsilon>0$ \\
$\{(i_1, i_2, ..., i_l)\in \mathbb{N}^l: i_1, i_2, ..., i_l\leq n,\ g(x, x_{i_1}, x_{i_2},..., x_{i_l})<\epsilon\}\\
\supseteq \{(i_1, i_2, ..., i_l)\in \mathbb{N}^l:  i_j\in A, i_1, i_2, ..., i_l\leq n,\ g(x, x_{i_1}, x_{i_2},..., x_{i_l})<\epsilon\}$,\\
and

\begin{align*}
&\lim_{n\to \infty}\frac{l!}{n^l}|\{(i_1, i_2, ..., i_l)\in \mathbb{N}^l: i_1, i_2, ..., i_l\leq n,\ g(x, x_{i_1}, x_{i_2},..., x_{i_l})<\epsilon\}|\\
&\geq  \lim_{n\to \infty}\frac{l!}{n^l}|\{(i_1, i_2, ..., i_l)\in \mathbb{N}^l: i_j\in A, i_1, i_2, ..., i_l\leq n,\ g(x, x_{i_1}, x_{i_2},..., x_{i_l})<\epsilon\}|\\
&=1.
\end{align*}

Hence $gs-\lim_{n\to \infty}x_n=x.$
\end{proof}
The following corollary is a direct consequence of Theorem 2.3.
\begin{cor}
In $g$-metric spaces, every statistically convergent sequence has a convergent subsequence .
\end{cor}

\begin{thm}
Every statistically convergent sequence is statistically $g$-Cauchy.  
\end{thm}
\begin{proof}
Let $\{x_n\}$ be a statistical $g$-Cauchy sequence in $g$-metric space $(X,g)$ and  $\epsilon>0$, then,
\begin{equation*}
\lim_{n\to \infty}\frac{l!}{n^l}|\{(i_1, i_2, ..., i_l)\in \mathbb{N}^l: i_1, i_2, ..., i_l\leq n,\ g(x, x_{i_1}, x_{i_2},..., x_{i_l})<\frac{\epsilon}{l(l+1)}\}|=1
\end{equation*} 
By the monotonicity condition for the $g$-metric and parts (4,7) of Theorem 1.1, it follows that,
\begin{equation*}
g(x_{i_{\epsilon}}, x_{i_1}, x_{i_2},..., x_{i_l})\leq\sum_{k=0}^{l}g(x_{i_k}, x,..., x)\leq \sum_{k=0}^{l}lg(x, x_{i_k}..., x_{i_k}).
\end{equation*}
So 
	\begin{align*}
	&\{(i_1, i_2, ..., i_l)\in \mathbb{N}^l: i_1, i_2, ..., i_l\leq n,\ g(x, x_{i_1}, x_{i_2},..., x_{i_l})<\frac{\epsilon}{l(l+1)}\}\\ 
	&\subseteq
	\{(i_1, i_2, ..., i_l)\in \mathbb{N}^l: i_1, i_2, ..., i_l\leq n,\ g(x_{i_\epsilon}, x_{i_1}, x_{i_2},..., x_{i_l})<\epsilon\}.
	\end{align*}
Therefore
\begin{equation*}
\lim_{n\to\infty}\frac{l!}{n^l}|\{(i_1, i_2, ..., i_l)\in \mathbb{N}^l: i_1, i_2, ..., i_l\leq n,\ g(x_{i_\epsilon}, x_{i_1}, x_{i_2},..., x_{i_l})<\epsilon\}|=1.
\end{equation*}
Thus, $\{x_n\}$ is a statistical $g$-Cauchy sequence in $(X, g)$.
\end{proof}
\begin{dfn}
Let $(X,g)$ be a $g$-metric space, if every statistically Cauchy sequence be statistically convergent, then $(X,g)$ is said to be {\em statistically complete}.
\end{dfn}
\begin{cor}
Every statistically complete $g$-metric space is complete.
\end{cor}
\begin{proof}
Let $(X,g)$ be a statistically complete $g$-metric space. Suppose $\{x_n\}$ be a Cauchy sequence in $(X,g)$, then it is statistically Cauchy sequence in $(X,g)$. Since $(X,g)$ is statistically complete so  $\{x_n\}$ is statistically convergent. By corollary 2.1 there is a subsequence  $\{x_{n_k}\}$ of  $\{x_n\}$ that converges to a point $x\in X$.\\
Since $\{x_n\}$ is Cauchy, hence, for $\epsilon >0$, there exist $N\in \mathbb{N}$ and $x_{i_{\epsilon}}\in \{x_n\}$ such that for $i_1, i_2, ...,i_l\geq N$ we have,
$$g(x_{i_{\epsilon}}, x_{i_1}, x_{i_2}, ..., x_{i_l})<\frac{\epsilon}{2l(l+1))}$$
On the other hands,  $\{x_{n_k}\}$ converges to $x$, Hence there exists $k_0\geq N$ such that for  $i_1, i_2, ...,i_l\geq k_0$,
$$g(x, x_{i_{n_1}}, x_{i_{n_2}}, ..., x_{i_{n_l}})<\frac{\epsilon}{2}.$$
For $i_1, i_2, ...,i_l\geq N$ and applying parts (3) and (4) of Theorem 1.1, it follows that,

\begin{align*}
g(x, x_{i_1}, x_{i_2}, ..., x_{i_l})&\leq g(x, x_{i_{\epsilon}}, x_{i_{\epsilon}},..., x_{i_{\epsilon}})+\sum_{j=1}^{l}g(x_{i_j}, x_{i_{\epsilon}}, x_{i_{\epsilon}},..., x_{i_{\epsilon}})\\
&\leq g(x, x_{n_{i_1}}, x_{n_{i_1}}, ..., x_{n_{i_1}})+l(g(x_{i_{\epsilon}}, x_{n_{i_1}}, x_{n_{i_1}}, ..., x_{n_{i_1}}))\\
&+\sum_{j=1}^{l}lg(x_{i_{\epsilon}},x_{i_j}, x_{i_j}, ..., x_{i_j})\\
&<\frac{\epsilon}{2}+l(\frac{\epsilon}{2l(l+1)})+l^2(\frac{\epsilon}{2l(l+1)})=\epsilon.
\end{align*}
\end{proof}

\section{Conclusion}
Recall that $G$-metric space is a special case of $g$-metric space with degree $l=2$. So, the main definitions and results of statistical forms of convergence for sequences was preferred to be applied for $g$-metric spaces. In this case, all results that have been discussed in this paper can be used for $G$-metric spaces specially.\\
On the other hand, if $l$ takes 1 in the definitions of statistical notion of convergence introduced in this paper, then these definitions exactly coincide with the statistical forms in ordinary metric spaces. Therefore, it can be said that this study is more comprehensive in terms of definitions and results.\\

\bibliographystyle{amsplain}

\end{document}